\documentclass[11pt]{amsart}
\usepackage{amsmath}
\usepackage{amssymb}
\usepackage{amsthm}
\usepackage{enumerate}
\usepackage[english]{babel}
\usepackage{hyperref}
\usepackage{xypic}

\newtheorem{thm}{Theorem}[section]
\newtheorem{lem}[thm]{Lemma}
\newtheorem{cor}[thm]{Corollary}
\newtheorem{prop}[thm]{Proposition}

\theoremstyle{definition}

\newtheorem{rem}[thm]{Remark}
\newtheorem{defn}[thm]{Definition}

\newtheorem{ex}[thm]{Example}

\newtheorem*{acknowledgments*}{Acknowledgments}

\numberwithin{equation}{section}

\theoremstyle{remark}
\newtheorem{notation}[thm]{Notation}
\newcommand\bk[1]{\left\langle #1 \right\rangle}

\mathchardef\ordinarycolon\mathcode`\: 
\def\vcentcolon{\mathrel{\mathop\ordinarycolon}} 
\providecommand*\coloneqq{\mathrel{\vcentcolon\mkern-1.2mu}=}

\def\Z{{\mathbb Z}} %integer
\def\R{{\mathbb R}} %real
\def\C{{\mathbb C}} %complex

\def\matrix{{\mathbb{M}}}

\def\K{{K}}

\def\HE{{H\!E}}

\def\alg{{\mathrm{alg}}}

\def\A{{\mathcal A}}
\def\H{{\mathcal H}}
\def\D{\mathcal D}
\def\AHD{{(\A, \H, \D)}}

\def\ch{\mathrm{Ch}}
\def\pch{\mathrm{Ch}^\mathrm{pert}}
\newcommand\JLO[2]{\bk{#1 \mid #2}}
\def\Ind{{\mathrm{Ind}}}

\def\dom{{\mathrm{dom}}}%domain
\def\ker{{\mathrm{ker}}}%domain
\def\tensor{\widehat{\otimes }}

\def\str{{\mathrm{Str}}}

\newcommand\half[1]{\frac{#1}{2}}

\begin{document}

\title{Multiplicativity of the JLO-character}
\author{Otgonbayar Uuye}
\date{January 11, 2010}
\address{
Department of Mathematical Sciences\\
University of Copenhagen\\
Universitetsparken 5\\
DK-2100 Copenhagen E\\
Denmark}
\email{otogo@math.ku.dk}
\urladdr{http://www.math.ku.dk/~otogo}
\maketitle

\begin{abstract}
We prove that the Jaffe-Lesniewski-Osterwalder character is compatible with the $A_{\infty}$-structure of Getzler and Jones.
\end{abstract}

\setcounter{section}{-1}
\section{Introduction}

Let $M$ be a closed, smooth manifold of even dimension. An order-one, odd, elliptic, pseudodifferential operator $D$ on $M$ gives rise to a $K$-homology class $[D] \in K_{0}(M)$. The celebrated Atiyah-Singer index theorem \cite{AS:1963, AS:1968,AS:1968b} computes the homological Chern character of $[D]$, in terms of the cohomological Chern character of the symbol class of $D$. 
There are many proofs known to date. The original one, as explained in \cite{P:1965}, proceeds as follows. The {\em bordsim invariance} and {\em multiplicativity} of the index map reduces the problem to the computation of the ``index genus'' $\Ind: \Omega^{\mathrm{SO}}_{\bullet}(KU) \to \Z$. Then deep results of Thom \cite{T:1954} and Conner-Floyd \cite{CF:1964} further reduces the problem, in effect, to the Hirzebruch signature theorem and the Bott periodicity theorem. The proof in \cite{AS:1968, AS:1968b} bypasses  the bordism computation altogether; here the {\em strict multiplicativity} (B3$'$) of the index map plays a crucial role.

In this paper, we study the multiplicative property\footnote{We do not yet have a good theory of bordism in noncommutative geometry; although preliminary results were announced in \href{http://www.math.uni-bonn.de/people/lesch/homepagenew/Papers/2009-LesMosPfl-Oberwolfach.pdf}{Matthias-Moscovici-Pflaum}.} of the index map in noncommutative geometry. In our setting, $\theta$-summable spectral triples play the role of $K$-homology classes and the JLO character replaces the homological Chern character. There are many other characters, especially if the spectral triple is finitely summable, but it seems that the JLO character is most compatible with the  exterior product operation. We show that the JLO character is compatible with the $A_{\infty}$-exterior product structure on entire chains (Theorem \ref{thm: A-inf and JLO}). The main idea goes back to \cite{GJP:1991,BG:1994}.

As a corollary, we construct a perturbation of the JLO character that is multiplicative at the chain level (Corollary~\ref{cor: multiplicativity of perturbed JLO}). Application to the index theory of transversally elliptic operators will appear elsewhere.

\begin{acknowledgments*} This paper is essentially contained in my thesis. I would like to thank my advisor Nigel Higson for his constant encouragement and guidance. I would also like to thank the NCG group at University of Copenhagen for support.
\end{acknowledgments*}
  
\section{Spectral triples}

\begin{defn}\label{defn spectral triple} A {\em spectral triple} $\AHD$  consists of a unital Banach algebra $\A$, a graded Hilbert space $\H$, equipped with a continuous, even representation of $\A$ and a densely defined, self-adjoint, odd operator $\D$ such that 
\begin{enumerate}
\item for any $a \in \A$, the commutator $[\D, a] \coloneqq \D a - a\D$ is bounded, that is, if $\dom(\D)$ is the domain of $\D$,  then $a \cdot \dom(\D) \subseteq \dom(\D)$ and $[\D, a]: \dom(\D) \to \H$ extends by continuity to a bounded operator on $\H$, and satisfies\footnote{It is enough to require that (\ref{norm req}) is satisfied up to a multiplicative constant.}
	\begin{equation}\label{norm req}
	||a|| + ||[\D, a]|| \le ||a||_{\A},
	\end{equation}
where $|| \cdot ||$ denotes the operator norm, and
\item\label{compact resolvent} the resolvents $(\D \pm i)^{-1}$ are compact.
\end{enumerate}
We say that $\AHD$ is {\em $\theta$-summable} if 
	\begin{enumerate}
\item[(3)]\label{theta} the operator $e^{-t\D^{2}}$ is of trace class for any $t > 0$.
\end{enumerate}
\end{defn}

\begin{rem}\label{rem compact resolvent} The compact resolvent condition (\ref{compact resolvent}) is equivalent to 
\begin{enumerate}
\item[(2')] the operator $e^{-t\D^{2}}$ is compact for any (or some) $t > 0$.
\end{enumerate}
Indeed, let 
	\begin{equation}
	\mathcal C  \coloneqq \{ f \in  C_{0}(\R) \mid \text{$f(\D)$ is compact}\}.
	\end{equation}
Then $\mathcal C$ is a closed ideal in $C_{0}(\R)$ and thus $e^{-tx^{2}} \in C_{0}(\R)$, $t > 0$ belongs to $\mathcal C$ iff  $\mathcal C = C_{0}(\R)$ iff $(x \pm i)^{-1}$ belong to $\mathcal C$.
\end{rem}
 
For practical purposes, it is useful to consider essentially self-adjoint operators. We say that $\AHD$ is a {\em pre-spectral triple} if $\A$ is a normed algebra not necessarily complete and $\D$ is required to be just essentially self-adjoint, in Definition~\ref{defn spectral triple}.
 
\begin{lem} Let $\AHD$ be a pre-spectral triple. Then $(\bar\A, \H, \bar\D)$ is a spectral triple, where $\bar\A$ denote the completion of $\A$, acting on $\H$ by continuous extension and $\bar\D$ denote the closure of $\D$.
\end{lem}
We call $(\bar\A, \H, \bar{\D})$ the {\em closure} of $\AHD$.
\begin{proof} Let $W^{1}$ denote the strong domain of $\bar \D$. We show that $\bar\A$ preserves $W^{1}$. %\[W^{1} \coloneqq \{\xi \in \H \mid \D\xi \in \H \text{ strongly}\}.\]
Let $a \in \bar\A$ and let $\xi \in W^{1}$. By definition, there exists a sequence $a_{n} \in \A$ converging to $a$ and a sequence $\xi_{m} \in \dom(\D)$ converging to $\xi$ such that $\D\xi_{m}$ converges to some $\eta \in \H$ as $n \to \infty$. Then, by (\ref{norm req}), the sequence $\overline{[\D, a_{n}]}$ is Cauchy and thus have a limit, which we denote $\overline{[\D, a]}$. Again using (\ref{norm req}), we see that the sequence $a_{n}\xi_{n} \in \dom(\D)$ converges to $a\xi \in \H$,  while $\D(a_{n}\xi_{n})$ converges to $\overline{[\D, a]}\xi + a\eta
 \in \H$ as $n \to \infty$. Hence $a\xi$ belongs to $W^{1}$.
 
It is clear that $[\bar\D, a]$ has a bounded extension on $\H$, namely $\overline{[\D, a]}$, which satisfy (\ref{norm req}), and that $\bar\D$ has compact resolvents.
\end{proof}

\begin{ex} Let $M$ be a closed manifold equipped with a smooth measure and let $S$ be a graded, smooth, Hermitian vector bundle over $M$. Let $\H \coloneqq L^{2}(M, S)$ denote the graded Hilbert space of $L^{2}$-sections of $S$. Let $\D$ be an odd, symmetric, elliptic pseudo-differential operator acting on the smooth sections $C^{\infty}(M, S)$ of $S$. We consider $\D$ as an unbounded operator on $\H$ with domain $C^{\infty}(M, S)$. Let $\A \coloneqq C^{\infty}(M)$ be the algebra of smooth functions acting on $\H$ by pointwise multiplication, equipped with the norm 	\begin{equation}
	||a||_{\A} \coloneqq ||a|| + ||[\D, a]||.
	\end{equation}
Then standard $\Psi$DO theory implies that $\AHD = (C^{\infty}(M), L^{2}(M, S), \D)$ is a pre-spectral triple (cf.\ \cite{Shu:2001}).
\end{ex}
 
 The following is well-known.
\begin{prop}\label{prop product} Let $(\A_{1}, \H_{1}, \D_{1})$ and $(\A_{2}, \H_{2}, \D_{2})$ be spectral triples. Let $\A \coloneqq \A_{1} \otimes_{\alg} A_{2}$ denote the algebraic tensor product, equipped with the projective tensor product norm $|| \cdot ||_{\pi}$, and let $\H \coloneqq \H_{1} \tensor \H_{2}$ denote the graded Hilbert space tensor product. Let 
	\begin{equation}
	\D \coloneqq \D_{1} \tensor 1 + 1 \tensor \D_{2}
	\end{equation}
be the operator with domain 
	\begin{equation}
	\dom(\D) \coloneqq \dom(\D_{1}) \tensor_{\alg} \dom(\D_{2}) \subseteq \H,
	\end{equation}
the algebraic graded tensor product. Then $(\A, \H, \D)$ is a pre-spectral triple.
\end{prop}
We write $\D_{1} \times \D_{2}$ for the closure of $\D$.
\begin{proof} 

Suppose $\xi_{1} \in \H_{1}$ and $\xi_{2} \in \H_{2}$ are analytical vectors for $\D_{1}$ and $\D_{2}$, respectively. Then $\xi_{1} \tensor \xi_{2}$ is a smooth vector for $\D$ and, for $t > 0$,
	\begin{align}
	\sum_{n = 0}^{\infty}\frac{||\D^{n}(\xi_{1} \tensor \xi_{2})||}{n!}t^{n} &\le \sum_{n = 0}^{\infty}\frac{1}{n!}\sum_{k=0}^{n}{k \choose n}||\D_{1}^{k}\xi_{1}||\,||\D_{2}^{n-k}\xi_{2}||t^{n}\\
	&= \left(\sum_{k = 0}^{\infty}\frac{||\D_{1}^{k}\xi_{1}||}{k!}t^{k}\right)\left(\sum_{m = 0}^{\infty}\frac{||\D_{2}^{m}\xi_{2}||}{m!}t^{m}\right).
	\end{align}
Hence, choosing $t > 0$ small, we see that $\xi_{1} \tensor \xi_{2}$ is an analytical  vector for $\D$. Since the finite linear combination of such elementary tensors is dense in $\dom(\D)$, Nelson's analytical vector theorem (cf.\ \cite[Theorem X.39]{RS:1975}) proves that $\D$ is essentially self-adjoint.

A similar argument shows that 
	\begin{equation}\label{eq product of exp}
	e^{-t(\D_{1} \times \D_{2})^{2}} = e^{-t\D_{1}^{2}} \tensor e^{-t\D_{2}^{2}}, \quad t>0,
	\end{equation}
hence, by Remark~\ref{rem compact resolvent}, $\D$ has compact resolvents.	

Finally, it is clear that $\A$ preserves the domain $\dom(\D)$. Moreover, if $a = \sum b_{i} \otimes c_{i}$ is an element of $\A$, then 
	\begin{equation}\label{eq derivation product}
	[\D, a] = \sum \left([\D_{1}, b_{i}] \tensor c_{i} + b_{i} \tensor [\D_{2}, c_{i}]\right)
	\end{equation}
has a bounded extension to $\H$ and the inequality 
	\begin{align}
	||\sum b_{i} \otimes c_{i}|| + ||[\D, \sum b_{i} \otimes c_{i}]|| &\le \sum ||b_{i}||\cdot ||c_{i}||\\
	&\hspace{-50pt}+\sum \left(||[\D_{1}, b_{i}]|| \cdot ||c_{i}|| + ||b_{i}|| \cdot ||[\D_{2}, c_{i}]||\right)\\
	& \le \sum ||b_{i}||_{\A_{1}} \cdot ||c_{i}||_{\A_{2}} \\
	& \le ||\sum b_{i} \otimes c_{i}||_{\pi}
	\end{align}
shows that (\ref{norm req}) is satisfied.	
\end{proof}

\begin{defn} Let $(\A_{1}, \H_{1}, \D_{1})$ and $(\A_{2}, \H_{2}, \D_{2})$ be spectral triples. We define their {\em product} as
	\begin{equation}\label{defn eq product}
	(\A_{1} \otimes_{\pi} \A_{2}, \H_{1} \tensor \H_{2}, \D_{1} \times \D_{2}),
	\end{equation}
where $\otimes_{\pi}$ denote the projective tensor product.
\end{defn}

\begin{rem}\label{rem product of st} By Proposition~\ref{prop product}, that the tripe (\ref{defn eq product}) is indeed a spectral triple. It follows from equation (\ref{eq product of exp}) that the product of $\theta$-summable spectral triples is again $\theta$-summable. Finally, note that taking product of spectral triples is {\em associative} under the natural identifications.
\end{rem}

\section{The JLO-character}

For this section, see \cite{JLO:1988,C:1988b,GS:1989,C:1991c} for details. Recall that associated to a spectral triple $\AHD$, there is an {\em index map}
	\begin{equation}
	\Ind_{\D}: K_{0}(\A) \to \Z, 
	\end{equation}
given by associating to an idempotent $p \in \A \otimes \matrix_{k}$, representing a class in $\K_{0}(\A)$, the Fredholm index of the Fredholm operator
	\begin{equation}
	p \D  p: p(\H^{0} \otimes \C^{k}) \to p(\H^{1} \otimes \C^{k}).
	\end{equation}
 Here $\matrix_{k}$ denote the algebra of $k \times k$ complex matrices.

For $\theta$-summable spectral triples, it can be computed ``homologically'', using the entire cyclic theory of Connes (cf.\ \cite{C:1988b}), as follows.  We follow the convention of \cite{GS:1989}. Let
	\begin{equation}
	C_{n}(\A) \coloneqq \A \otimes_{\pi} (\A/\C)^{\otimes_{\pi}n},\quad n \in \Z_{\ge 0},
	\end{equation}
and let $C_{\bullet}(\A)$ denote the completion of $\bigoplus_{n = 0}^{\infty} C_{n}(\A)$ with respect to the collection of norms
	\begin{equation}
	||\oplus_{n} \alpha_{n}||_{\lambda} \coloneqq \sum_{n = 0}^{\infty} \frac{\lambda^{n}||\alpha_{n}||_{\pi}}{\sqrt{n!}}, \quad \lambda \in \Z_{\ge 1}.
	\end{equation}
Let $b$ and $B$ denote the Hochschild and Connes boundary maps on $C_{\bullet}(\A)$ respectively. The {\em entire cyclic homology} group $\HE_{\bullet}(\A)$ is defined as the homology of the complex $(C_{\bullet}(\A), b + B)$. The {\em entire cyclic cohomology} group $\HE^{\bullet}(\A)$ is defined using the topological dual $C^{\bullet}(\A)$ of $C_{\bullet}(\A)$.

\begin{notation} Let 
	\begin{equation}
	\Sigma^{n} \coloneqq \{ t = (t^{1}, \dots, t^{n}) \mid 0 \le t^{1} \le \dots \le t^{n} \le 1\} \subset [0, 1]^{n}
	\end{equation}
denote the standard $n$-simplex equipped with the standard Lebesgue measure $dt = dt^{1}\dots dt^{n}$ with volume $\frac{1}{n!}$.
\end{notation}	
\begin{defn}[{Jaffe-Lesniewski-Osterwalder \cite{JLO:1988}}]\label{defn JLO} Let $\AHD$ be a $\theta$-summable spectral triple. Denote $\Delta \coloneqq \D^{2}$ and define $da \coloneqq [\D, a]$, $a \in \A$.

For $(a^{0}, \dots, a^{n}) \in \A \otimes (\A/\C)^{\otimes n}$ and $t \in \Sigma^{n}$, we define
	\begin{equation}
	\label{JLO}
	\JLO{a^{0}, \dots, a^{n}}{t}_{\D} \coloneqq \str(a^{0}e^{-t^{1}\Delta}da^{1} e^{-(t^{2} - t^{1})\Delta} \dots d a^{n}e^{-(1 - t^{n})\Delta}),
	\end{equation}
where $\str$ is the super-trace on $\H$, and	
	\begin{equation}
	\ch^{n}_{\D}(a^{0}, \dots, a^{n}) \coloneqq \int_{\Sigma^{n}}\!\!\JLO{a^{0}, \dots, a^{n}}{t}_{\D}dt.
	\end{equation} 
\end{defn}

Then, $\ch^{\ast}_{\D}$ defines an element of $\HE^{0}(\A)$ called the {\em JLO-character} of $\AHD$ and satisfies the abstract index formula
	\begin{equation}
	\langle \ch^{\ast}_{\D}, \ch_{\ast}(e) \rangle = \Ind_{\D}(e), \quad e \in K_{0}(\A),
	\end{equation}
where $\ch_{\ast}(e) \in \HE_{0}(\A)$ denote the entire cyclic homological Chern character of $e \in K_{0}(\A)$. See \cite{GS:1989} for details.

\section{Multiplicativity}

The following multiplicative property of the index map is a folklore.

\begin{prop}\label{prop index map is multiplicative} Let $(\A_{1}, \H_{1}, \D_{1})$ and $(\A_{2}, \H_{2}, \D_{2})$ be spectral triples. Then the following diagram is commutative:
	\begin{equation}
	\xymatrix{
	\K_{0}(\A_{1}) \otimes \K_{0}(\A_{2}) \ar[r]\ar[d]^{\Ind_{\D_{1}} \otimes \Ind_{\D_{2}}} & \K_{0}(\A_{1} \otimes_{\pi} \A_{2}) \ar[d]^{\Ind_{\D_{1} \times \D_{2}}}\\
\Z \otimes \Z \ar@{=}[r]^{\times} & \Z
	}.
	\end{equation}
\end{prop}
 
\begin{proof} Let $e_{i} \in \K_{0}(\A_{i})$, $i \in \{1, 2\}$, be given and let $e_{1} \otimes e_{2}$ denote their product in $K_{0}(\A_{1} \otimes_{\pi} \A_{2})$. We need to show that
	\begin{equation}
	\Ind_{\D_{1}}(e_{1}) \cdot \Ind_{\D_{2}}(e_{2}) = \Ind_{\D_{1} \times \D_{2}}(e_{1} \otimes e_{2}).
	\end{equation}
As in the proof of \cite[Theorem D]{GS:1989}, we may assume that $\A_{i}$ is an involutive Banach algebra acting involutively on $\H_{i}$ and $e_{i}$ is represented by a self-adjoint idempotent $p_{i} \in \A_{i}$. Then $(p_{i}\A_{i}p_{i}, p_{i}\H_{i}, p_{i}\D_{i}p_{i})$ is a spectral triple. Moreover we can easily see that 
	\begin{equation}
	p_{1} \tensor p_{2} (\D_{1} \times \D_{2}) p_{1} \tensor p_{2} = p_{1}\D_{1}p_{1} \times p_{2} \D_{2}p_{2}
	\end{equation}
 on $p_{1} \tensor p_{2} \cdot \H_{1} \tensor \H_{2} = p_{1}\H_{1} \tensor p_{2}\H_{2}$. Hence, we may assume that $p_{i} = 1 \in \A_{i}$.
 
 Let $P_{i}$ denote the projection onto $\ker (\D_{i}^{2})$ and let $P$ denote the projection onto $\ker (\D_{1} \times \D_{2})^{2}$. Taking the limit $t \to \infty$ in (\ref{eq product of exp}), we see that
 	\begin{equation}
	P = P_{1} \tensor P_{2}.
	\end{equation}
Hence
	\begin{equation}
	\Ind(\D_{1} \times \D_{2}) = \str(P) = \str(P_{1}) \str(P_{2}) = \Ind(\D_{1}) \Ind(\D_{2}).
	\end{equation}
The proof is complete.		
\end{proof}

Now we study the multiplicative property of the JLO-character. We start by defining an ``$A_{\infty}$-exterior product structure'' on entire cyclic chains, following \cite{GJ:1990, GJP:1991}. 

First, the Hochschild shuffle product is defined as follows.
\begin{defn}
Let $p$, $q \in \Z_{\ge 1}$ be natural numbers and let $S(p, q)$ denote the set $\{(1, 1), \dots, (1, p), (2, 1), \dots, (2, q)\}$, ordered lexicographically, that is, 
	\begin{equation}\label{order Spq}
	(1, 1)  < \dots < (1, p) < (2, 1) < \dots < (2, q).
	\end{equation}
A {\em permutation} $\chi$ of $S(p, q)$is called a {\em $(p, q)$-shuffle}\index{shuffle} if 
	\[\chi(1, 1) < \dots < \chi(1, p) \quad\text{and}\quad \chi(2, 1) < \dots < \chi(2, q).\] 
\end{defn}

Let $n = p + q$. Then the ordering (\ref{order Spq}) gives an identification of the set $S(p, q)$ with the set $\{1, \dots, n\}$ and we use this identification to let permutations of $S(p, q)$ act on $\{1, \dots, n\}$.

Let the group of permutations of $\{1, \dots, n\}$ act on $C_{n}(\A)$ by
	\[\chi(a^{0}, a^{1},\dots , a^{n}) \coloneqq (-1)^{\chi} (a^{0}, a^{\chi^{-1}(1)}, \dots , a^{\chi^{-1}(n)}).\] 

\begin{defn} Let $\alpha = (a^{0}, a^{1},\dots , a^{p}) \in C_{p}(\A_{1})$ and $\beta = (b^{0}, b^{1}, \dots , b^{q}) \in C_{q}(\A_{2})$ be elementary tensors. The {\em shuffle product}\index{product!shuffle} $\alpha \times \beta \in C_{p+q}(\A_{1} \otimes \A_{2})$ is defined as
	\begin{align*}
	\alpha \times \beta \coloneqq \sum_{(p, q)\text{-shuffles}}\chi(a^{0} \otimes b^{0}, a^{1} \otimes 1, \dots, a^{p} \otimes 1, 1 \otimes b^{1}, \dots, 1 \otimes b^{q}).
	\end{align*}
\end{defn}

\begin{rem}\label{rem shuffle extends}
We note that the shuffle product extends to 
	\[C_{\bullet}(\A_{1}) \otimes C_{\bullet}(\A_{2}) \to C_{\bullet}(\A_{1} \otimes_{\pi} \A_{2}).\]
Indeed, if $\alpha \in C_{\bullet}(\A_{1})$ and $\beta \in C_{\bullet}(\A_{2})$, then for any $\lambda$, $\mu \in \Z_{\ge 1}$ satisfying $\mu \ge \sqrt{2} \lambda$, we have
	\begin{align}
	||\alpha \times \beta||_{\lambda} &= \sum_{n \ge p \ge 0} \frac{\lambda^{n}||\alpha_{p} \times \beta_{n - p}||_{\pi}}{\sqrt{n!}}\\
	&\le \sum_{n \ge p \ge 0} \frac{\lambda^{n}{n \choose p}||\alpha_{p}||_{\pi} ||\beta_{n - p}||_{\pi}}{\sqrt{n!}}\\
	&\le ||\alpha||_{\mu} \cdot ||\beta||_{\mu}.
	\end{align}
Here we used the fact that the number of $(p, n - p)$-shuffles is ${n \choose p} \le 2^{n}$.
\end{rem}

It is well-known and easy to see that the shuffle product is associative and compatible with the Hochschild boundary map $b$. However, the Connes boundary map $B$ is {\em not} a derivation with respect to the shuffle product and we fix this by perturbing the shuffle product by a term called the cyclic shuffle product $B_{2}$. The new product we obtain is not associative, but only homotopy associative. In fact, there exists a sequence of operations
	\begin{equation}
	B_{r}: C_{p_{1}}(\A_{1}) \otimes \dots \otimes C_{p_{r}}(\A_{r}) \to C_{r + p_{1} + \dots + p_{r}}(\A_{1} \otimes \dots \otimes \A_{r})
	\end{equation}
such that $B_{1} = B$ and $B_{2} = \text{cyclic shuffle product}$ and $B_{3}$ ``controls'' the failure of associativity et cetera, defined using the following combinatorial device.

\begin{defn} 
Let $r \in \Z_{\ge 1}$ and $p_{1}, \dots, p_{r} \in \Z_{\ge 0}$. 
Let $C(p_{1}, \dots , p_{r})$ be the set $\left\{{0 \choose 1}, \dots, {p_{1} \choose 1}, \dots, {0 \choose r}, \dots, {p_{r} \choose r}\right\}$, ordered lexicographically,	that is ${l_{1} \choose k_{1}} \le {l_{2} \choose k_{2}}$ if and only if $k_{1} < k_{2}$ or $k_{1} = k_{2}$ and $l_{1} \le l_{2}$. This ordering gives an identification of $C(p_{1}, \dots, p_{r})$ with the set $\{1, \dots, r + p_{1} + \dots p_{r}\}$.

A $(p_1, \dots, p_r)$-cyclic shuffle is a permutation $\sigma$ of the set $C(p_{1}, \dots, p_{r})$ such that
	\begin{enumerate}
	\item $\sigma{0 \choose i_{1}} < \sigma{0 \choose i_{2}}$ if $i_{1} <i_{2}$ and
	\item for each $1 \le i \le r$, there is a number $0 \le j_{i} \le p_{i}$ such that
		\begin{equation*}\sigma{j_{i}\choose i} < \dots < \sigma{p_{i} \choose i} < \sigma{0 \choose i} < \dots < \sigma{j_{i} - 1 \choose i}.
		\end{equation*}
	\end{enumerate}
\end{defn}

\begin{defn}
For $\alpha_{i} = (a_{1}^{0}, \dots, a_{1}^{p_{i}}) \in C_{p_{i}}(\A_{i})$, we define $B_{r}(\alpha_{1}, \dots, \alpha_{r}) \in C_{r + p_{1} + \dots + p_{r}}(\A_{1} \otimes \dots \otimes \A_{r})$ as
	\begin{equation}
	B_{r}(\alpha_{1}, \dots, \alpha_{r}) \coloneqq \sum_{\sigma} \sigma(1, a_{1}^{0}, \dots, a_{1}^{p_{1}}, \dots, a_{r}^{0}, \dots, a_{r}^{p_{r}}),
	\end{equation}	
where $a_{i}^{j} \in \A_{i}$ is considered an element of $\A_{1} \otimes \dots \otimes \A_{r}$ via $a_i \mapsto 1_{\A_{1}} \otimes \dots \otimes 1_{\A_{i-1}} \otimes a_{i} \otimes 1_{\A_{i+1}} \otimes \dots  \otimes 1_{\A_{r}}$ and the summation is over all $(p_{1}, \dots, p_{r})$-cyclic shuffles. For $\alpha \in C_{\bullet}(\A_{1})$ and $\beta \in C_{\bullet}(\A_{2})$, we also write
	\begin{equation}
	\alpha \times' \beta \coloneqq B_{2}(\alpha, \beta)
	\end{equation}
and call $\times'$ the {\em cyclic shuffle product}.	
\end{defn}	

\begin{rem} If $\A$ is commutative, we recover the $B_{k}$ terms of the $A_{\infty}$-structure of Getzler and Jones \cite{GJ:1990} using the multiplication map.
\end{rem}

One motivation for the definition of the shuffles and the cyclic shuffles is the following. 

We let permutations on $\{1, \dots, n\}$ act on $[0, 1]^{n}$ by
 	\begin{equation}\chi(t^{1}, \dots, t^{n}) \coloneqq (t^{\chi^{-1}(1)}, \dots, t^{\chi^{-1}(n)}).
	\end{equation}

Then for any element $t \in [0, 1]^{n}$, such that all the entries are distinct, there exists a {\em unique} permutation $\chi$ such that $\chi(t)$ is an element of $\Sigma^{n}$, {\em i.e.}\ entries of $\chi(t)$ are in increasing order. Therefore, permutations give a decomposition of $[0, 1]^{n}$ into $n$-simplices. The shuffles and the cyclic shuffles give decompositions of product simplices. 

\begin{lem}[{Getzler-Jones-Petrack \cite{GJP:1991}}]\label{lem simplicial decomposition}
\begin{enumerate}
\item\label{item shuffle} Let $p$ and $q$ be natural numbers. For a $(p, q)$-shuffle $\chi$, define 
	\begin{align}\Sigma(\chi) &\coloneqq \{(s, t) \in \Sigma^{p} \times \Sigma^{q} \subset [0, 1]^{p+q} \mid \chi(s, t) \in \Sigma^{p+q}\}.
	\end{align}

Then, $\Sigma(\chi)$ is a $(p+q)$-simplex and, up to a set of measure zero, $\Sigma^{p} \times \Sigma^{q}$ is the disjoint union of the sets $\Sigma(\chi)$, $\chi$ shuffle.

\item\label{item cyclic shuffle} Similarly, for a cyclic shuffle $\sigma$, define
	\begin{align}
	\Sigma(\sigma) &\coloneqq \{(s, t_{1}, \dots, t_{r}) \mid \sigma(s + t_{1} + \dots  + t_{r}) \in \Sigma^{r + p_{1}+ \dots + p_{r}}\},
	\end{align}
where	$(s, t_{1}, \dots, t_{r}) \in \Sigma^{r} \times \Sigma^{p_{1}} \times \dots \times \Sigma^{p_{r}} \subset [0, 1]^{r + p_{1}+ \dots + p_{r}}$ and $s + t^{1} + \dots + t^{r}$ denote the $(r + p_{1} + \dots + p_{r})$-tuple
	\begin{equation}
	(s^{1}, s^{1} + t_{1}^{1}, \dots s^{1} + t_{1}^{p_{1}}, \dots, s^{r}, s^{r} + t_{r}^{1}, \dots, s^{r} + t_{r}^{p_{r}})
	\end{equation}
considered modulo $1$.

Then, $\Sigma(\sigma)$ is a $(r + p_{1}+ \dots + p_{r})$-simplex and, up to a set of measure zero, $\Sigma^{r} \times \Sigma^{p_{1}} \times \dots \times \Sigma^{p_{r}}$ is the disjoint union of the sets $\Sigma(\sigma)$, $\sigma$ cyclic shuffle.
\end{enumerate}
\qed
\end{lem}

\begin{rem} It follows from Lemma~\ref{lem simplicial decomposition}(\ref{item cyclic shuffle}) that the number of $(p_{1}, \dots, p_{r})$-cyclic shuffles is
	\begin{equation}
	{r + p_{1} + \dots + p_{r} \choose r, p_{1}, \dots, p_{r}} = \frac{(r + p_{1} + \dots + p_{r})!}{r!p_{1}! \dots p_{r}!}.
	\end{equation}
An argument similar to Remark~\ref{rem shuffle extends}, proves that the operation $B_{r}$ extend to
	\begin{equation}
	B_{r}: C_{\bullet}(\A_{1}) \otimes_{\pi} \dots \otimes_{\pi} C_{\bullet}(\A_{r}) \to C_{r + \bullet}(\A_{1} \otimes_{\pi} \dots \otimes_{\pi} \A_{r}). 
	\end{equation}
\end{rem}

For $B = B_{1}$, we have the following. First note that the expression (\ref{JLO}) makes sense for $a^{0}$ in $\A + [\D, \A]$. We write
	\begin{equation}
	\JLO{d\alpha}{t}_{\D} \coloneqq \JLO{da^{0}, a^{1}, \dots, a^{n}}{t}_{\D}
	\end{equation}
for $\alpha = (a^{0}, \dots, a^{n}) \in C_{n}(\A)$ and $t \in \Sigma^{n}$.	

\begin{prop}[{\cite[Lemma 2.2(2)]{GS:1989}}]\label{GJ} Let $\AHD$ be a $\theta$-summable spectral triple. Then  
	\begin{equation}
	B\ch_{\D}(\alpha) = \int_{\Sigma^{n}}\JLO{d\alpha}{t}dt, 
	\end{equation}
for $\alpha \in C_{n}(\A)$.	
\qed
\end{prop}

Following is the analogue of \cite[Proposition 4.1, 4.2]{GJP:1991} and \cite[Theorem 3.2]{BG:1994}. Part (\ref{cyclic shuffle JLO}) below uses Proposition~\ref{GJ} and can be considered an extension of it to the case $r \ge 2$.

\begin{thm}\label{thm: A-inf and JLO}
\begin{enumerate}
\item\label{shuffle JLO}
 Let $(\A_{1}, \H_{1}, \D_{1})$ and $(\A_{2}, \H_{2}, \D_{2})$ be $\theta$-summable spectral triples. Then
	\begin{equation}
	\ch_{\D_{1} \times \D_{2}}(\alpha_{1} \times \alpha_{2}) = \ch_{\D_{1}}(\alpha_{1})\ch_{\D_{2}}(\alpha_{2}),
	\end{equation}
for $\alpha_{1} \in C_{\bullet}(\A_{1})$ and $\alpha_{2} \in C_{\bullet}(\A_{2})$.
\item\label{cyclic shuffle JLO} Let $(\A_{i}, \H_{i}, \D_{i})$, $1 \le i \le r$, be $\theta$-summable spectral triples. Then
	\begin{equation}\label{eq chBr}
	\ch_{\D_{1} \times \D_{2} \times \dots \times \D_{r}}B_{r}(\alpha_{1}, \dots, \alpha_{r}) = \frac{1}{r!}B\ch_{\D_{1}}(\alpha_{1}) \dots B\ch_{\D_{r}}(\alpha_{r}),
	\end{equation}
for $\alpha_{i} \in C_{\bullet}(\A_{i})$, $1 \le i \le r$. In particular, for $r = 2$, 
	\begin{equation}
	\ch_{\D_{1} \times \D_{2}}(\alpha_{1} \times' \alpha_{2}) = \half{1}B\ch_{\D_{1}}(\alpha)B\ch_{\D_{2}}(\alpha_{2}).
	\end{equation}
\end{enumerate}
\end{thm}

Note that Remark~\ref{rem product of st} shows that the theorem is well-posed. 

\begin{proof} 
First we prove (\ref{shuffle JLO}).
Let $\alpha_{1} = (a^{0}, a^{1}, \dots, a^{p}) \in C_{p}(\A_{1})$ and $\alpha_{2} = (b^{0}, b^{1}, \dots, b^{q}) \in C_{q}(\A_{2})$ and let $\gamma = (c^{0, 0}, c^{1, 1}, \dots, c^{1, p}, c^{2, 1}, \dots, c^{2, q}) \in C_{p+q}(\A_{1} \otimes \A_{2})$ denote the element
	\begin{equation}
	 (a^{0} \otimes b^{0}, a^{1} \otimes 1, \dots, a^{p} \otimes 1, 1\otimes b^{1}, \dots, 1 \otimes b^{q}).
	\end{equation}
Then using (\ref{eq product of exp}) and (\ref{eq derivation product}), we see that 
	\begin{equation}
	e^{-t(\D_{1} \times \D_{2})^{2}}[\D_{1} \times \D_{2}, c^{i, j}] = \begin{cases}e^{-t\Delta_{1}}[\D_{1}, a^{j}] \tensor e^{-t\Delta_{2}}, & i = 1\\ e^{-t\Delta_{1}} \tensor e^{-t\Delta_{2}}[\D_{2}, b^{j}], & i = 2 \end{cases},
	\end{equation}
for $t > 0$. Now it is straightforward to check that for any $(p, q)$-shuffle $\chi$ and $(s, t) \in \Sigma(\chi) \subset \Sigma^{p} \times \Sigma^{q}$ with $u = \chi(s, t) \in \Sigma^{p+q}$,
	\begin{align}
	\JLO{\chi(\gamma)}{u}_{\D_{1} \times \D_{2}} = \JLO{\alpha_{1}}{s}_{\D_{1}}\cdot \JLO{\alpha_{2}}{t}_{\D_{2}}.
	\end{align}
Integrating over $\Sigma(\chi)$ and summing over all the $(p, q)$-shuffles $\chi$, we get the result, using Lemma~\ref{lem simplicial decomposition}(\ref{item shuffle}). 

The proof of (\ref{cyclic shuffle JLO}) is similar. Let $\gamma$ denote the element
	\begin{equation}
	(1, a_{1}^{0}, \dots, a_{1}^{p_{1}}, \dots, a_{r}^{0}, \dots, a_{r}^{p_{r}})
	\end{equation}
in $C_{r + p_{1} + \dots + p_{r}}(\A_{1} \otimes \dots \otimes \A_{r})$.	
Then for any cyclic $(p_{1}, \dots, p_{r})$-shuffle $\sigma$, 
and $(s, t_{1}, \dots, t_{r}) \in \Sigma(\sigma) \subset \Sigma^{r} \times \Sigma^{p_{1}} \times \dots \times \Sigma^{p_{r}}$ with $v = \sigma(s + t_{1} + \dots + t_{r}) \in \Sigma^{r + p_{1} + \dots + p_{r}}$, we have	
	\begin{align*}
	\JLO{\sigma(\gamma)}{v}_{\D_{1} \times \dots \times \D_{r}} = \JLO{d\alpha_{1}}{t_{1}}_{\D_{1}}\cdot \dots \cdot \JLO{d\alpha_{r}}{t_{r}}_{\D_{r}}.
	\end{align*}
Now Lemma~\ref{lem simplicial decomposition}(\ref{item cyclic shuffle}) and Proposition \ref{GJ} completes the proof. The factor $1/r!$ in (\ref{eq chBr}) is the volume of $\Sigma^{r}$.
\end{proof}

\begin{defn} Let $\AHD$ be a $\theta$-summable spectral triple. We define the {\em perturbed JLO cocycle}\index{perturbed JLO character}\index{JLO!perturbed} as
	\begin{equation}\label{eq pch}
	\pch_{\bullet} = \ch_{\bullet} + \frac{1}{\sqrt{2}}B\ch_{\bullet-1}
	\end{equation}
\end{defn}
\begin{rem}
\begin{enumerate}
\item The perturbed JLO character is a cocycle: 
	\begin{align*}(b+B)\pch_{\bullet} &= (b+B)(\ch_{\bullet} + 2^{-\half{1}}B\ch_{\bullet-1})\\
	&= 2^{-\half{1}}bB\ch_{\bullet-1}\\
	&= -2^{-\half{1}}Bb\ch_{\bullet-1}\\
	&= -2^{-\half{1}}B(b + B)\ch_{\bullet-1}\\
	&= 0.
	\end{align*}
\item The perturbed JLO character has mixed parity: $\ch_{\bullet}$ is even and $B\ch_{\bullet_{-1}}$ is odd. The pairing with $K_{0}$ depends only on the even part, hence $\ch$ and $\pch$ have the same pairing with $K_{0}(\A)$.
\end{enumerate}
\end{rem}

\begin{cor}\label{cor: multiplicativity of perturbed JLO} Let $(\A_{1}, \H_{1}, \D_{1})$ and $(\A_{2}, \H_{2}, \D_{2})$ be $\theta$-summable spectral triples. Then for $\alpha \in C_{\bullet}(\A_{1})$ and $\beta \in C_{\bullet}(\A_{2})$
	\begin{equation}
	\pch_{\D_{1} \times \D_{2}}(\alpha \times \beta + \alpha\times'\beta) = \pch_{\D_{1}}(\alpha)\cdot \pch_{\D_{2}}(\beta).
	\end{equation}
\end{cor}

\begin{proof} 
For a $\theta$-summable spectral triple $\AHD$, we write  
	\begin{equation}
	\delta(a^{0}, a^{1}\dots, a^{p}) \coloneqq \frac{1}{\sqrt{2}}([\D, a^{0}], a^{1}, \dots, a^{p}).
	\end{equation}
Note that $[\D, a^{0}]$ does {\em not} necessarily belong to $\A$, but this causes no trouble.

Then we can write
	\begin{align}
	\pch_{\D} &= \ch_{\D}\circ (1 + \delta).
	\end{align}

Now we write $\delta_{1}$, $\delta_{2}$ and $\delta_{12}$ for the $\delta$ corresponding to the spectral triples $(\A_{1}, \H_{1}, \D_{1})$, $(\A_{2}, \H_{2}, \D_{2})$ and $(\A_{1}, \H_{1}, \D_{1}) \times (\A_{2}, \H_{2}, \D_{2})$, respectively. Then, using (\ref{eq derivation product}) we see that
	\begin{align}
	\delta_{12}(\alpha \times \beta) =\delta_{1}(\alpha) \times \beta + \alpha \times \delta_{2}(\beta).
	\end{align}

Moreover $\delta_{12}(\alpha \times' \beta) = 0$ because all the summands start with the term $[\D_{1} \times \D_{2}, 1 \otimes 1] = 0$.

Therefore,
	 \begin{align*}\pch_{\D_{1} \times \D_{2}}&(\alpha \times \beta + \alpha\times'\beta)\\ &= \ch_{\D_{1} \times \D_{2}} \left((1 + \delta_{12})(\alpha \times \beta + \alpha\times'\beta)\right)\\
	 &=\ch_{\D_{1} \times \D_{2}}(\alpha \times \beta + \alpha\times'\beta + \delta_{1}(\alpha) \times \beta + \alpha \times \delta_{2}(\beta))\\
	 &= \ch_{\D_{1}}(\alpha)\ch_{\D_{2}}(\beta) + \ch_{\D_{1}}(\delta_{1}\alpha)\ch_{\D_{2}}(\delta_{2}\beta)\\
	  &\qquad+ \ch_{\D_{1}}(\delta_{1}\alpha)\ch_{\D_{2}}(\beta) + \ch_{\D_{1}}(\alpha)\ch_{\D_{2}}(\delta_{2}\beta)\\
	 &= \ch_{\D_{1}}\left((1 + \delta_{1})(\alpha)\right) \ch_{\D_{2}}\left((1 + \delta_{2})(\beta)\right)\\
	 &= \pch_{\D_{1}}(\alpha) \cdot \pch_{\D_{2}}(\beta),
	\end{align*}
by Theorem~\ref{thm: A-inf and JLO}.
\end{proof}

\begin{cor} For $theta$-summable spectral triples, the perturbed JLO character implements the diagram in Lemma~\ref{prop index map is multiplicative}.
\qed
\end{cor}

\bibliographystyle{amsalpha}
\bibliography{Biblio-Database}

\end{document}